\documentclass[a4,11pt]{amsart}
\usepackage{amsfonts}
\usepackage{mathrsfs}
\usepackage{amsthm,amsxtra}
\usepackage{amssymb}
\usepackage{amsmath}
\usepackage{amscd}
\usepackage{enumerate}
\usepackage{hyperref}
\usepackage{t1enc}
\usepackage[mathscr]{eucal}
\usepackage{indentfirst}
\usepackage{graphicx}
\graphicspath{ {images/} }
\usepackage[margin=2.9cm]{geometry}
\usepackage{tikz-cd}
\usepackage{tikz}

\theoremstyle{plain}

\newtheorem{theorem}{Theorem}[section]

\newtheorem{lemma}[theorem]{Lemma}

\theoremstyle{definition}

\theoremstyle{remark}

\DeclareMathOperator{\Cor}{Cor}

\DeclareMathOperator{\Sym}{Sym}

\DeclareMathOperator{\Aut}{Aut}

\DeclareMathOperator{\Out}{Out}

\DeclareMathOperator{\MA}{M\Gamma}
\DeclareMathOperator{\Id}{Id}
\DeclareMathOperator{\MM}{\mathcal{M}}

\makeatletter
\DeclareFontFamily{OMX}{MnSymbolE}{}
\DeclareSymbolFont{MnLargeSymbols}{OMX}{MnSymbolE}{m}{n}
\SetSymbolFont{MnLargeSymbols}{bold}{OMX}{MnSymbolE}{b}{n}
\DeclareFontShape{OMX}{MnSymbolE}{m}{n}{
    <-6>  MnSymbolE5
   <6-7>  MnSymbolE6
   <7-8>  MnSymbolE7
   <8-9>  MnSymbolE8
   <9-10> MnSymbolE9
  <10-12> MnSymbolE10
  <12->   MnSymbolE12
}{}
\DeclareFontShape{OMX}{MnSymbolE}{b}{n}{
    <-6>  MnSymbolE-Bold5
   <6-7>  MnSymbolE-Bold6
   <7-8>  MnSymbolE-Bold7
   <8-9>  MnSymbolE-Bold8
   <9-10> MnSymbolE-Bold9
  <10-12> MnSymbolE-Bold10
  <12->   MnSymbolE-Bold12
}{}

\let\llangle\@undefined
\let\rrangle\@undefined
\DeclareMathDelimiter{\llangle}{\mathopen}%
                     {MnLargeSymbols}{'164}{MnLargeSymbols}{'164}
\DeclareMathDelimiter{\rrangle}{\mathclose}%
                     {MnLargeSymbols}{'171}{MnLargeSymbols}{'171}
\makeatother

\title{On trialities and their absolute geometries}
\author{Dimitri Leemans}\thanks{This research was made possible thanks to an Action de Recherche Concert\'ee grant from the Communaut\'e Fran\c caise Wallonie-Bruxelles.}
\address{Universit\'e Libre de Bruxelles, D\'epartement de Math\'ematique, C.P.216 - Alg\`ebre et Combinatoire, Boulevard du Triomphe, 1050 Brussels, Belgium, Orcid number 0000-0002-4439-502X.}
\curraddr{}
\email{Leemans.Dimitri@ulb.be}
\urladdr{}

\author{Klara Stokes}
\address{Department of Mathematics and Mathematical Statistics, Ume\aa\; University,
901 87 Ume\aa, Sweden, Orcid number 0000-0002-5040-2089.}
\email{klara.stokes@umu.se}

\author{Philippe Tranchida}
\address{Universit\'e Libre de Bruxelles, D\'epartement de Math\'ematique, C.P.216 - Alg\`ebre et Combinatoire, Boulevard du Triomphe, 1050 Brussels, Belgium, Orcid number 0000-0003-0744-4934.}
\curraddr{}
\email{tranchida.philippe@gmail.com}
\urladdr{}

\date{\today}
{}
\begin{document}

\maketitle
\begin{abstract}
We introduce the notion of moving absolute geometry of a geometry with triality and show that, in the classical case where the triality is of type $(I_\sigma)$ and the absolute geometry is a generalized hexagon, the moving absolute geometry also gives interesting flag-transitive geometries with Buekenhout diagram
\begin{center}

    \begin{tikzpicture}
    
   \filldraw[black] (-3,0) circle (2pt)  node[anchor=north]{};
   \filldraw[black] (3,0) circle (2pt)  node[anchor=north]{};
    \draw (-3,0) -- (3,0) node [midway, above = 3pt, fill=white]{$5 \; \; \; \; \; 3 \; \; \; \; \; 6$};
    \end{tikzpicture}
    \end{center}
for the groups $G_2(k)$ and $^3D_4(k)$, for any integer $k \geq 2$. We also classify the classical absolute geometries for geometries with trialities but no dualities coming from maps of Class III with automorphism group $L_2(q^3)$, where $q$ is a power of a prime. We then investigate the moving absolute geometries for these geometries, illustrating their interest in this case.
\end{abstract}

\section{Introduction}
The notion of triality in geometry is an important concept that dates back from papers of Study (see~\cite[Page 435]{Porteous}, see also~\cite{Study1} and~\cite{Study2}\footnote{ See \url{http://neo-classical-physics.info/uploads/3/4/3/6/34363841/study-analytical_kinematics.pdf} 
 for an english translation of~\cite{Study2}.}).
His idea was to use a quadric living in a seven dimensional projective space to describe motions. He used the quadric $Q$ defined by the equation $X_0X_7+X_1X_6+X_2X_5+X_3X_4=0$ in homogeneous coordinates and showed that it features a nice property that
Cartan later on called triality~\cite{Cartan}, defining the phenomenon in the broader context of Lie groups. 
Freudenthal further explored triality in the context of Lie algebras \cite{Freudenthal}.
  
  In 1959, Tits classified in~\cite{tits1959trialite} the trialities with absolute points of the quadric $Q$.
He observed that the set of absolute points and the incident lines fixed by the triality form a rank two incidence geometry with the property that the girth of the incidence graph is two times the diameter. This discovery motivated his definition of generalized polygon -- the content of the appendix of~\cite{tits1959trialite}.

But what if, instead of taking lines fixed by the triality, we take lines that are moved by the triality?
This leads us to introduce the notion of a moving absolute geometry. Roughly speaking, the moving absolute geometry associated to a triality is a geometry of rank two with points and lines. The points are the absolute points (as in the classical absolute of Tits) and the lines are moving lines, meaning they are not fixed by the triality, but we require that they contain at least two absolute points.
It turns out that this moving absolute geometry also gives interesting rank two geometries in the classical case. We show that the moving absolute geometries in the classical cases are flag-transitive geometries for the groups $G_2(k)$ and $^3D_4(k)$, where $f$ is the cardinality of the underlying field and $k$ is $f$ for $G\cong G_2(k)$ and $f$ is the cube of $k$ for $G\cong$  $^3D_4(k)$. Moreover, we prove that their Buekenhout diagrams are as follows.

\begin{center}

    \begin{tikzpicture}
    
    \filldraw[black] (-3,0) circle (2pt)  node[anchor=north]{$k$};
    \filldraw[black] (3,0) circle (2pt)  node[anchor=north]{$(k+1)f^2-1$};
    \draw (-3,0) -- (3,0) node [midway, above = 3pt, fill=white]{$5 \; \; \; \; \; 3 \; \; \; \; \; 6$};
    \end{tikzpicture}
    \end{center}

We also revisit the geometries with trialities constructed in~\cite{leemans2022incidence}. We show that their absolute geometries are basically unions of paths of length two and we compute the moving absolute geometries of some of them, finding more interesting geometries. The absolute geometries of thin geometries, or more generally of geometries with rather small residues of rank $2$, will often be very poorly connected (see Theorem \ref{thm: absoluteMaps} for an example of such behaviour). To get richer geometries from a triality in that context, the moving absolute geometry is then a good candidate.

\section{Preliminaries} \label{sec:prelim}
\subsection{ Incidence and Coset Geometries}

To their core, most of the geometric objects of interest to mathematicians are composed of elements together with some relation between them. This very general notion is made precise by the notion of an incidence system, or an incidence geometry.

    A triple $\Gamma = (X,*,\tau)$ is called an \textit{incidence system} over $I$ if
    \begin{enumerate}
        \item $X$ is a set whose elements are called the \textit{elements} of $\Gamma$
        \item $*$ is a symmetric and reflexive relation on $X$. It is called the \textit{incidence relation} of $\Gamma$.
        \item $\tau$ is a map from $X$ to $I$, called the \textit{type map} of $\Gamma$, such that distinct elements $x,y \in X$ with $x * y$ satisfy $\tau(x) \neq \tau(y)$. Elements of $\tau^{-1}(i)$ are called the elements of type $i$.
    \end{enumerate}

The \textit{rank} of $\Gamma$ is the cardinality of the type set $I$.
A \textit{flag} in an incidence system $\Gamma$ over $I$ is a set of pairwise incident elements. The type of a flag $F$ is $t(F)$, that is the set of types of the elements of $F.$ A \textit{chamber} is a flag of type $I$. An incidence system $\Gamma$ is an \textit{incidence geometry} if all its maximal flags are chambers.

Let $F$ be a flag of $\Gamma$. An element $x\in X$ is {\em incident} to $F$ if $x*y$ for all $y\in F$. The \textit{residue} of $\Gamma$ with respect to $F$, denoted by $\Gamma_F$, is the incidence system formed by all the elements of $\Gamma$ incident to $F$ but not in $F$. The \textit{rank} of a residue is equal to rank$(\Gamma)$ - $|F|$.

The \textit{incidence graph} of $\Gamma$ is a graph with vertex set $X$ and where two elements $x$ and $y$ are connected by an edge if and only if $x * y$. Whenever we talk about the distance between two elements $x$ and $y$ of a geometry $\Gamma$, we mean the distance in the incidence graph of $\Gamma$ and simply denote it by $d_\Gamma(x,y)$, or even $d(x,y)$ if the context allows.

Let $\Gamma = \Gamma(X,*,\tau)$ be an incidence geometry over the type set $I$. A correlation of $\Gamma$ is a bijection $\phi$ of $X$ respecting the incidence relation $*$ and such that, for every $x,y \in X$, if $\tau(x) = \tau(y)$ then $\tau(\phi(x)) = \tau(\phi(y))$. If, moreover, $\phi$ fixes the types of every element (i.e $\tau(\phi(x)) = \tau(x)$ for all $x \in X$), then $\phi$ is said to be an automorphism of $\Gamma$. The \emph{type} of a correlation $\phi$ is the permutation it induces on the type set $I$. A correlation of type $(i,j)$ is called a duality if it has order $2$. A correlation of type $(i,j,k)$ is called a triality if it has order $3$. The group of all correlations of $\Gamma$ is denoted by $\Cor(\Gamma)$ and the automorphism group of $\Gamma$ is denoted by $\Aut(\Gamma)$. Remark that $\Aut(\Gamma)$ is a normal subgroup of $\Cor(\Gamma)$ since it is the kernel of the action of $\Cor(\Gamma)$ on $I$.

Incidence geometries can be obtained from a group $G$ together with a set $(G_i)_{i \in I}$ of subgroups of $G$. 

    The \emph{coset geometry} $\Gamma(G,(G_i)_{i \in I})$ is the incidence geometry over the type set $I$ where:
    \begin{enumerate}
        \item The elements of type $i \in I$ are right cosets of the form $G_i \cdot g$, $g \in G$.
        \item The incidence relation is given by non empty intersection. More precisely, the element $G_i \cdot g$ is incident to the element $G_j \cdot k$ if and only if $i\neq j$ and $G_i \cdot g \cap G_j \cdot k \neq \emptyset$.
    \end{enumerate}

A lot of properties of incidence geometries, such as connectedness, residual connectedness, residues, flag-transitivity, and so on, can be translated to group theoretical properties of $(G,(G_i)_{i \in I})$ (see \cite{buekenhout2013diagram} for a more detailed exposition).

Francis Buekenhout introduced in~\cite{buek} a new diagram associated to $\Gamma$. His idea was to associate to each rank two residue a set of three integers giving information on its incidence graph.
Let $\Gamma$ be a rank $2$ geometry. We can consider $\Gamma$ to have type set $I = \{P,L\}$, standing for points and lines. The {\em point-diameter}, denoted by $d_P(\Gamma) = d_P$, is the largest integer $k$ such that there exists a point $p \in P$ and an element $x \in \Gamma$ such that $d(p,x) = k$. Similarly the {\em line-diameter}, denoted by $d_L(\Gamma) = d_L$, is the largest integer $k$ such that there exists a line $l \in L$ and an element $x \in \Gamma$ such that $d(l,x) = k$. Finally, the \textit{gonality} of $\Gamma$, denoted by $g(\Gamma) = g$ is half the length of the smallest circuit in the incidence graph of $\Gamma$.

If a rank $2$ geometry $\Gamma$ has $d_P = d_L = g = n$ for some natural number $n$, we say that it is a \textit{generalized $n$-gon}. Generalized $2$-gons are also called generalized digons. They are in some sense trivial geometries since all points are incident to all lines. Their incidence graphs are complete bipartite graphs. Generalized $3$-gons are projective planes.

Let $\Gamma$ be a geometry over $I$.  The \textit{Buekenhout diagram} (or diagram for short) $D$ for $\Gamma$ is a graph whose vertex set is $I$. Each edge $\{i,j\}$ is labeled with a collection $D_{ij}$ of rank $2$ geometries. We say that $\Gamma$ belongs to $D$ if every residue of rank $2$ of type $\{i,j\}$ of $\Gamma$ is one of those listed in $D_{ij}$ for every pair of $i \neq j \in I$. In most cases, we use conventions to turn a diagram $D$ into a labeled graph. The most common convention is to not draw an edge between two vertices $i$ and $j$ if all residues of type $\{i,j\}$ are generalized digons, and to label the edge $\{i,j\}$ by a natural integer $n$ if all residues of type $\{i,j\}$ are generalized $n$-gons. It is also common to omit the label when $n=3$.
If the edge $\{i,j\}$ is labeled by a triple $(d_{ij},g_{ij},d_{ji})$ it means that every residue of type $\{i,j\}$ had $d_P = d_{ij}, g = g_{ij}, d_L = d_{ji}$. We can also add information to the vertices of a diagram. 
We can label the vertex $i$ with the number $n_i$ of elements of type $i$ in $\Gamma$. Moreover, if for all flags $F$ of co-type $i$, we have that $|\Gamma_F| = s_i +1$, we will also label the vertex $i$ with the integer $s_i$.

Let $\Gamma$ be an incidence geometry and let $\phi$ be a correlation of $\Gamma$. The action of this correlation $\phi$ on $\Gamma$ will induce a new geometry, called the absolute geometry of $\Gamma$ with respect to $\phi$.

    The \textit{absolute geometry} of $\Gamma$ with respect to $\phi$ is the incidence geometry $\Gamma_\phi = (X_\phi, *_\phi, \tau_\phi)$ over $J$ where
    \begin{enumerate}
        \item The set $J$ is the collection of all $\phi$-orbits $K$ on $I$ for which there exist invariant flags of type $K$;
        \item The set $X_\phi$ is the set of all non empty $\phi$-invariant flags of $\Gamma$;
        \item The relation $*_\phi$ is determined by $F *_\phi G$ if and only if $F \cup G$ is a flag of $\Gamma$;
        \item The function $\tau_\phi \colon X_\phi \to J$ is the map assigning to a minimal $\phi$-invariant flag $F$ the set of $\phi$-orbits in $\tau(F)$.
    \end{enumerate}

This concept of absolute geometry motivated Tits in \cite{tits1959trialite} to define generalized polygons.

\subsection{Maps}

 A \textit{map} $\mathcal{M}$ is a $2$-cell embedding of a graph into a closed surface. In other words, a map is composed of a set $V = V(\mathcal{M})$ of vertices, a set $E = E(\mathcal{M})$ of edges and finally a set $F = F(\mathcal{M})$ of faces, which are the simply connected components obtained by cutting the surface along $V \cup E$. A \textit{flag} $F$ of a map $\mathcal{M}$ is a triple $\{v,e,f\}$ with $v \in V, e \in E, f \in F$ and such that each element is incident to the two others. 

An \textit{automorphism} of a map $\mathcal{M}$ is a permutation of its elements preserving the three sets $V,E$ and $F$ and sending incident pairs to incident pairs and non-incident pairs to non-incident pairs. The group of all automorphisms of a map $\mathcal{M}$ is denoted by $\Aut(\mathcal{M})$. A map is said to be \textit{reflexible} if $\Aut(\mathcal{M})$ has only one orbit on the set of flags of $\mathcal{M}$.

One can always define three operation on the set of flags of a map $\mathcal{M}$. Let $F =\{v,e,f\}$ be a flag. Then there is exactly one flag $F_0$ of $\mathcal{M}$ that coincides with $F$ on the elements $e$ and $f$ but has a different vertex. Similarly there is a unique flag $F_1$, respectively $F_2$, coinciding with $F$ except for $e$, respectively $f$. The three operations, denoted by $\rho_0, \rho_1$ and $\rho_2$, send each flag to its unique $i$-adjacent flag, $i = 0,1,2$. It is easily seen that $\rho_0$ and $\rho_2$ always commute. In other words, there is always an action of the Coxeter group $C = \langle \rho_0, \rho_1 , \rho_2 \mid \rho_0^2 = \rho_1^2 = \rho_2^2 = (\rho_0 \rho_2)^2 = e  \rangle \cong V_4 * C_2$ on the set of flags of a map $\mathcal{M}$. If the map $\mathcal{M}$ is reflexible, its automorphism group $\Aut(\mathcal{M})$ is the quotient of $C$ by the stabilizer of a flag of $\mathcal{M}$. Conversely, given a finite quotient of $C$, it is possible to reconstruct a map $\mathcal{M}$ from the action of $C$ on its flags. This gives a correspondence between finite quotients of $C$ acting on sets and maps $\mathcal{M}$ on closed surfaces. This correspondance is functorial, in the sense that it sends $C$-equivariant maps to morphisms of maps, and vice-versa.
A reflexible map then corresponds to an epimorphism of $C$ to a finite group $G$.

Given a base map $\mathcal{M}$, there are some operations than one can apply to $\mathcal{M}$ to obtain new maps. One of them is called the \emph{dual operator} $D$ and comes from the classical notion of duality, on polytopes for example. The dual map $D(\mathcal{M})$ is obtained from $\mathcal{M}$ by switching the roles of vertices and faces. From the group theoretic perspective, the operator $D$ then simply exchanges $\rho_0$ and $\rho_2$. Another such operator is the \emph{Petrie dual operator} $P$. A \emph{Petrie path} is a "left-right" path in $\MM$. This means that the path turns once left at a vertex and next time it turns right, and so on, until it comes back to the starting vertex. The operator $P$ fixes the vertices and the edges of the map $\MM$ but replaces the faces. The map $D(\MM)$ is obtained by deleting the faces of $\MM$ and then, for every Petrie path, gluing a disk with boundary corresponding to the Petrie path. This corresponds to fixing $\rho_1$ and $\rho_2$ and sending $\rho_0$ to $\rho_0\rho_2$. One can then also consider composition of these two operators $D$ and $P$. Wilson showed that these two operators and their compositions form a group $\Sigma \cong S_3$ \cite{wilson1979operators}. The operators $D \circ P$ and $P \circ D$ are thus of order $3$. We will refer to these two operators or order $3$ as \emph{trialities} of the map $\MM$ and we will say that the map $\MM$ has trialities if $\MM$ is isomorphic to $ D \circ P (\MM)$ and $P \circ D (\MM)$.

Jones and Thornton showed that the group $\Sigma$ is the outer automorphism group ${\rm Out}(\Gamma)={\rm Aut}(\Gamma)/{\rm Inn}(\Gamma)\cong S_3$ \cite{jones1983operations}. 
The action of the six operators in $\Sigma$ on the generators of $C$ is showed in Figure \ref{classes}.

\begin{figure}
\begin{center}
\begin{picture}(135,160)
\put(50,150){$(\rho_0,\rho_1,\rho_2)$}
\put(0,100){$(\rho_2,\rho_1,\rho_0)$}
\put(75,140){\line(-2,-1){60}}
\put(25,130){$D$}
\put(75,140){\line(2,-1){60}}
\put(125,130){$P$}
\put(100,100){$(\rho_0\rho_2,\rho_1,\rho_2)$}
\put(15,90){\line(0,-1){30}}
\put(0,70){$P$}
\put(135,90){\line(0,-1){30}}
\put(140,70){$D$}
\put(0,50){$(\rho_0\rho_2,\rho_1,\rho_0)$}
\put(100,50){$(\rho_2,\rho_1,\rho_0\rho_2)$}
\put(15,40){\line(2,-1){60}}
\put(25,20){$D$}
\put(135,40){\line(-2,-1){60}}
\put(125,20){$P$}
\put(50,00){$(\rho_0,\rho_1,\rho_0\rho_2)$}
\end{picture}
\caption{Wilson's operations on the monodromy group of a map.}\label{classes}
\end{center}
\end{figure}
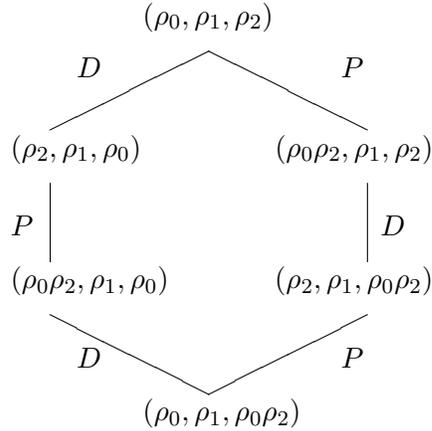

For a given map $\MM$, we can thus consider the set of $6$ maps obtained from $\MM$ by applying the operators of $\Sigma$. Some of these maps may be isomorphic to $\MM$ and some of them may not be. Wilson divided regular maps into four classes. A map $\MM$ is of Class I if $\Sigma(\MM)$ is formed of $6$ non isomorphic maps. It is of Class II is $\Sigma(\MM)$ splits into three pairs of isomorphic maps, of Class III if it splits into two triples of isomorphic maps and of Class IV if all $6$ maps are pairwise isomorphic. 

Jones and Thornton showed that every finite map has a finite reflexible cover of Class IV \cite{jones1983operations}. Richter, \v{S}ir\'a\v{n} and Wang proved that there is a Class IV map of every even valency~\cite{richter2012self}. 
This was later extended to odd valency $\geq 5$ by Fraser, Jeans and \v{S}ir\'a\v{n} \cite{fraser2018regular}. 
The kaleidoscopic maps due to Archdeacon, Conder and \v{S}ir\'a\v{n} are also, by definition, of Class IV \cite{archdeacon2014trinity}.

The first sporadic example of a map of Class III was constructed by Wilson in 1979 \cite{wilson1979operators}. 
It seemed to him at the time that maps of Class III are rare. As reported by Jones and Poulton, Conder found more sporadic examples of Class III maps in a computer search in 2006, but over 30 years passed before Jones and Poulton~\cite{jones2010maps} produced an infinite family of reflexible maps of Class III with automorphism group ${\rm L_2}(2^{3n})$, for $n$ a positive integer, thereby extending Wilson's  first example. In the same article Jones and Poulton also constructed maps of Class III as covers of other maps of Class III as well as parallel products.
In a recent paper Abrams and Ellis-Monaghan also constructed non-reflexible maps of Class III \cite{abrams2022new}. Finally, Leemans and Stokes also constructed \cite{leemans2022incidence} an infinite family of reflexible maps of Class III directly the simple groups $L_2(q^3)$ with $q = p^n$. The trialities of these maps come from the existence of the Frobenius automorphism.

\subsection{Quadric of type $D_4$ in $\mathbf{P}\mathbf{G}(7,\mathbb{F})$}\label{subsec:conic}

Let $\mathbf{Q}$ be an hyperbolic quadratic set in a projective space $\mathbf{P} = \mathbf{P}\mathbf{G}(7,\mathbb{F})$ of dimension $7$ over a field $\mathbb F$. The maximal subspaces of $\mathbf{Q}$ are of dimension $3$, as the index of $\mathbf{Q}$ is $4$. If we want to be more concrete, we can choose a set of homogeneous coordinates $\{ X_0,X_1,\cdots, X_7\} $ for $\mathbf{P}$ and fix the quadric $\mathbf{Q}$ to have equation 
\begin{equation} \label{eq:1}
X_0X_4 + X_1X_5 +X_2X_6 + X_3X_7 = 0
\end{equation}

We can define an equivalence relation on the set $M$ of maximal subspaces of $\mathbf{Q}$ as follows: for two subspaces $m,n \in M$, we set $m \equiv n$ if $m \cap n$ is of odd dimension. This relation is obviously reflexive and symmetric, and it can be shown to be also transitive. Moreover, there are exactly two equivalence classes of maximal subspaces, denoted by $M_1$ and $M_2$. If we suppose $\mathbf{Q}$ to have equation~(\ref{eq:1}), representatives for $M_1$ and $M_2$ can be taken to be the $3$-spaces obtained by setting $X_0 = X_2 = X_4 = X_6 =0$ and $X_1 = X_3 = X_5 =X_7 = 0$ respectively. We call the points of $\mathbf{Q}$ the $0$-points, the elements of $M_1$ as $1$-points and the elements of $M_2$ the $2$-points. 

We can then define a geometry $\Gamma$ of rank four on $\{0,1,2,3\}$.
The lines of $\mathbf{Q}$ are the elements of type 3.
The $i$-points (where $i= 0,1,2$) are the elements of type $i$. Incidence is given by symmetrized inclusion whenever it makes sense, and for a $1$-point $m_1$ and a $2$-point $m_2$, we set $m_1  * m_2$ if they intersect in a plane.
The geometry $\Gamma$ thus obtained can then be shown to have Buekenhout diagram $D_4$, see Figure \ref{D4}.

\begin{figure}
\begin{center}
\begin{tikzpicture}[scale = 1.5]
\label{D4}
\filldraw[black] (0,0) circle (2pt)  node[anchor=south]{};
\filldraw[black] (1,0) circle (2pt) node[anchor=west]{};
\filldraw[black] (-0.5,0.866) circle (2pt) node[anchor=east]{};
\filldraw[black] (-0.5,-0.866) circle (2pt) node[anchor=east]{};

\draw (1,0) -- (0,0) -- (-0.5,-0.866);
\draw (0,0) -- (-0.5,0.866);

\end{tikzpicture}
\caption{$D_4$ diagram.}   
\end{center}
\end{figure}
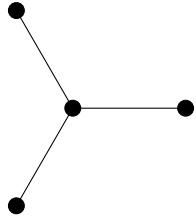

This geometry $\Gamma$ admits trialities that permute the $i$-points. Let $\alpha$ be such a triality. 
We can then consider the absolute geometry $\Gamma_\alpha$ of $\Gamma$ with respect to $\alpha$. We will call the points and the lines of $\Gamma_\alpha$ the \emph{absolute points} and \emph{absolute lines}. In \cite{tits1959trialite},Tits showed that, as long as there exists a cycle of absolute lines, the absolute $\Gamma_\alpha$ is a generalized hexagon.

Let $\sigma$ be an automorphism of the field $\mathbb{F}$ such that $\sigma^3 = \Id$. Tits classified the projective type of trialities $\tau$ of a projective plane $\pi$ over $\mathbb{F}$ into four categories, denoted by $(I_\sigma), (II), (III) $ and $(III)^\pm$ (see \cite{tits1959trialite}, section $2$). Surprisingly, there is a strong relation between trialities $\tau$ of $\pi$ and trialities $\alpha$ of $\mathbf{Q}$, that we briefly sketch here (for more details, see \cite{tits1959trialite}, section $4$).

We adopt the convention that upper case letters will designate points and lower case letters will designate lines. Also, if $P$ and $Q$ are two points, the line going through them will be designated by $(PQ)$. 
Suppose $P$ is a point of $\mathbf{Q}$ and let $P^\alpha$ and $P^{\alpha^2}$ be the associated incident $3$-spaces. We define $ P\omega$ to be the intersection $P^\alpha \cap P^{\alpha^2}$. If $P$ is not an absolute point, then $P\omega$ is a point. Instead, if $P$ is an absolute point, $ P\omega$ is then a plane, and we will sometimes refer to the plane $P \omega$ as the plane associated to $P$. Any plane that is the associated plane of some absolute point $P$ will be called a {\em special plane}. Similarly, any point of $\mathbf{Q}$ contained in a special plane is called a \emph{special point}. Notice that the plane $P \omega$ is spanned by any two absolute lines through $P$. This is true because all absolute points in $P^\alpha \cup P^{\alpha^2}$ are contained in $P \omega$.

Let $Q$ be a point of $\mathbf{Q}$ which is not special. Since $Q^\alpha$ and $ Q\omega$ are incident, it follows that $Q^{\alpha^2}$ and $( Q\omega)^\alpha$ are also incident. We define $Q\pi_1$ to be the plane $Q^{\alpha^2} \cap ( Q\omega)^\alpha$ and $Q\pi_2$ to be the plane $Q^\alpha \cap ( Q\omega)^{\alpha^2}$.

Under these notations, the triality $\alpha$ induces a map from the lines  of $\mathbf{Q}$ passing through $ Q\omega$ and contained in $Q^{\alpha^2}$ to the lines contained in $(Q\omega)^\alpha$ and passing by $Q$. We can then obtained a collineation $\alpha_Q$ from $Q \pi_1$ to itself as follows:

\begin{equation*}
    \alpha_Q(P) = Q\pi_1 \cap (P \cdot Q\omega)^\alpha
\end{equation*}
for any $P \in Q\pi_1$ and where $(P \cdot Q\omega)$ designates the line through $P$ and $Q\omega$.

Using the classification of trialities of a projective plane mentioned above, we can then say that a triality $\alpha$ of $\mathbf{Q}$ is of type $(I_\sigma), (II)$ or $(III)$ if there exists a non special point $Q$ of $\mathbf{Q}$ such that $\alpha_Q$ is of type $(I_\sigma), (II)$ or $(III)$, respectively. Tits showed  (see \cite{tits1959trialite}, section $5$) that all trialities $\alpha$ of type $(I_\sigma)$ with $\sigma \neq \Id$ are projectively equivalent and that all other trialities $\alpha$ are projectively equivalent to either a triality of type $(I_{\Id} )$ or a triality of type $(II)$. In this paper we will only consider the case where $\alpha$ is of type $(I_\sigma)$. While this is in many sense the more general case, it could be interesting to figure out what happens in the case where $\alpha$ is of type $(II)$.

We finish this section by a few useful tools and properties of $\mathbf{Q}$.

As in~\cite[Section 2.4.6]{van2012generalized},
let $V$ be an eight-dimensional vector space over $\mathbb{F}$. The fact that the points and the two types of $3$-spaces of $\mathbf{Q}$ play the same role can be expressed by the existence of a trilinear form $\mathcal{T} \colon V \times V \times V \to \mathbb{K}$. This form $\mathcal{T}$ is characterized by the fact that two points $(X,Y)$ of $\mathbf{Q}$ represent a $0$-point and a $1$-point of $\mathbf{Q}$ that are incident if and only if $\mathcal{T}(X,Y,Z)$ is identically zero as a function of $Z$. The same is true for any permutation of the letters $X,Y$ and $Z$.
This trilinear form has the following explicit description:

\begin{align}\label{trilinear}
\begin{split}
\mathcal{T}(X,Y,Z) =&   \begin{vmatrix}
     X_0& X_1 & X_2\\ 
     Y_0& Y_1 & Y_2\\
     Z_0& Z_1 & Z_2
\end{vmatrix}
+ \begin{vmatrix}
     X_4& X_4 & X_4\\ 
     Y_5& Y_5 & Y_5\\
     Z_6& Z_6 & Z_6
\end{vmatrix}
\\
& + X_3(Z_0Y_4 + Z_1Y_5 +Z_2Y_6) + X_7(Y_0Z_4 +Y_1Z_5 +Y_2Z_6) \\
& + Y_3(X_0Z_4 + X_1Z_5 +X_2Z_6) + Y_7(Z_0X_4 +Z_1X_5 +Z_2X_6) \\
& + Z_3(Y_0X_4 +Y_1X_5+Y_2X_6) + Z_7(X_0Y_4 +X_1Y_5 +X_2Y_6) \\
& - X_3Y_3Z_3 - X_7Y_7Z_7
\end{split}
\end{align}

When the triality is of type $(I_{\Id})$, it is well known that the absolute points are exactly the intersection of $\mathbf{Q}$ with the hyperplane of equation $X_3 + X_7 = 0$. We can thus substitute $X_7$ for $X_3$ and work with the parabolic quadric $\mathbf{Q'}$ in $\mathbf{P}'=\mathbf{P}\mathbf{G}(6,\mathbb{F})$ of equation

$$
X_0X_4+X_1X_5+X_2X_6 = X_3^2
$$
It can then be shown (see \cite[Section 2.4.13]{van2012generalized} for example) that the absolute lines of $\Gamma_\alpha$ have Grassmann coordinates satisfying the following six linear equations: 

\begin{equation}\label{Grassmann}
\begin{split}
X_{12} &= X_{34}, \qquad\qquad\qquad X_{54} = X_{32}, \qquad\qquad\qquad X_{20} = X_{35},\\
X_{65} &= X_{30}, \qquad\qquad\qquad X_{01} = X_{36}, \qquad\qquad\qquad X_{46} = X_{31},
\end{split}
\end{equation}
and conversely, every line on $\mathbf{Q}$ whose Grassmann coordinates satisfy (\ref{Grassmann}) is an absolute line of $\Gamma_\alpha$.
These equations will be used later to verify the existence of some lines of a new rank two geometry we will introduce in the next section as the moving absolute geometry of $\Gamma$.

Quadrics $\mathbf{Q}$ have an important property, called the \textit{all-or-one property}. This means that if $l$ is a line of $\mathbf{Q}$ and $P$ is a point of $\mathbf{Q}$ not contained in $l$, then there is either a unique line of $\mathbf{Q}$ passing through $P$ and intersecting $l$ or all lines passing through $P$ and intersecting $l$ are in $\mathbf{Q}$. This will often be useful in the proofs of the next section.

There is a natural group acting on both the absolute and the moving absolute geometry in this context. It is the group $G$ of collineations of $\mathbf{P}$ that preserves the triality $\alpha$. If $\alpha$ is a triality of type $(I_\sigma)$, as we will assume from now on, Tits showed that if $\sigma$ is the identity, then the group $G$ is of type $G_2$ and if $\sigma$ is not the identity then it is a twisted Lie group of type $D_4$, noted by $^3D_4$ (see~\cite[Section 8]{tits1959trialite}).

\section{moving absolute geometries}\label{sec:movingabsolute}

Let $\Gamma$ be a geometry with diagram $D_4$ (see figure \ref{D4}) and let $\alpha$ be a triality of $\Gamma$. The triality $\alpha$ must fix the central vertex of the diagram, and we will call an element that belongs to that fixed type \emph{a line}. We also choose one of the remaining types of $\Gamma$ and call elements belonging to that type \emph{points}. In this situation, we can define another type of absolute geometry for $\Gamma$ which has the same vertex set as the classical absolute geometry but considers lines that are moved by the triality instead of fixed lines. We will call such lines \emph{moving lines}, in contrast with the lines fixed by the triality which we will call \emph{absolute lines}.

    The \textit{moving absolute geometry} of $\Gamma$ with respects to $\alpha$ is the point-line geometry $\MA_\alpha$ over $J = \{P,L\}$ where
    \begin{enumerate}
   
        \item The points, called {\em absolute points}, are the points $p \in \Gamma$ such that $p, \alpha(p)$ and $\alpha^2(p)$ form a flag of $\Gamma$; each point has type $P$;
        \item The lines are the lines of $\Gamma$ that are not fixed by $\alpha$ and that contain at least two absolute points; each line has type $L$;
        \item Incidence is given by the incidence in $\Gamma$, that is, a point $p$ is incident to a line $l$ if $p$ and $l$ are incident in $\Gamma$. 
    \end{enumerate}
Remark that the definition of $\MA_\alpha$ does not depend, up to isomorphism, on the choice of points in $\Gamma$. Indeed suppose that we decide that the points of $\Gamma$ are now the elements of type $\alpha(P)$ instead. If $p$ and $p'$ are on a line $l$ in $\Gamma$, then $\alpha(p)$ and $\alpha(p')$ are incident to the line $\alpha(l)$. Therefore, both the lines and the incidence of $\MA_\alpha$ do not depend of the choice of points in $\Gamma$.

Whenever the triality $\alpha$ that we are working with is clear from context, we simplify the notation and talk about $\MA$ instead of $\MA_\alpha$. 

We now examine the moving absolute geometry of a quadric of type $D_4$ in a $7$-dimensional projective space.

\subsection{Quadrics of type $D_4$}

Recall that $\mathbf{Q}$ is a quadric in a $7$-dimensional projective space 
over a field $\mathbb{F}$ and that $\sigma$ is an automorphism of order $1$ or $3$ of $\mathbb{F}$ such that $\alpha$ is a triality of type $(I_\sigma)$ (see section \ref{subsec:conic}). 
Let $\mathbb{K}$ be the subfield of $\mathbb{F}$ fixed by $\sigma$. We will denote by $f$ and $k$ respectively the cardinalities of $\mathbb{F}$ and $\mathbb{K}$. Note that $f = k$ if $\sigma $ is the identity and $ f = k^3$ if $\sigma$ is not the identity. 

The main goal of this section is to prove that the moving absoulte $\MA$ in these settings is a geometry on two types with $d_P= 5, g =3$ and $d_l = 6$. The first results rules out the existence of some lines in $\mathbf{Q}$ and the proof is following the ideas of the proof of~\cite[Theorem 2.4.4]{van2012generalized}. Since the classical absolute geometry $\Gamma_\alpha$ is a generalized hexagon, there exist hexagons in $\mathbf{Q}$ whose vertices are all absolute points and whose lines are all absolute lines. We will call such an hexagon an \textit{absolute hexagon}.

\begin{lemma} \label{lem:oppositeVertices}
    Let $\mathbf{H}$ be an absolute hexagon. The line joining two opposite vertices of $\mathbf{H}$ is never a line of the conic $\mathbf{Q}$.
\end{lemma}

\begin{proof}
    Let $P$ be a vertex of $\mathbf{H}$ and let $l$ and $m$ be the two absolute lines of $\mathbf{H}$ meeting at $P$. Then the plane $P \omega$ is spanned by $l$ and $m$. Since the choice of $P$ was arbitrary, we conclude that the line between any two vertices at distance two of each other in $\mathbf{H}$ is always a line of $\mathbf{Q}$, and thus also a line of $\MA$ since it cannot be an absolute line. See Figure \ref{fig:MovingHexagon}, where examples of such lines are drawn in blue dashed lines. Now suppose that the line between a vertex $P$ of $\mathbf{H}$ and its opposite vertex $Q$ is a line of the conic $\mathbf{Q}$. Using the all-or-one property of $\mathbf{Q}$, we conclude that all the lines between $P$ and another point of $\mathbf{H}$ are in $\mathbf{Q}$. 
    Moreover, if we let $S$ be the space spanned by the $4$ lines of $\mathbf{H}$ not containing $P$, then using the all-or-one property again, we can conclude that all the lines between $P$ and $S$ are in $\mathbf{Q}$. Hence the whole hexagon $\mathbf{H}$ must be contained in a subspace $U$ of $\mathbf{H}$. Then $U$ must have dimension $3$ or $2$. If $U$ has dimension $3$, we can assume that $U = P^\alpha$. But we could have followed the same reasoning starting with $Q$, concluding that $U = Q^\alpha$, a contradiction. If $U$ has dimension $2$, it means the whole hexagon $\mathbf{H}$ is contained in a plane but then there must be $3$ absolute lines forming a triangle, which is not possible since the absolute geometry is a generalized hexagon.
\end{proof}
\begin{figure}
\centering
\begin{tikzpicture}

\filldraw[black] (4,0) circle (2pt)  node[anchor=west]{};
\filldraw[black] (-4,0) circle (2pt) node[anchor=east]{};
\filldraw[black] (2,3.4) circle (2pt) node[anchor=east]{};
\filldraw[black] (-2,3.4) circle (2pt) node[anchor=south]{$Q$};
\filldraw[black] (2,-3.4) circle (2pt) node[anchor=north]{$P$};
\filldraw[black] (-2,-3.4) circle (2pt) node[anchor=east]{};
\filldraw[black] (3,1.7) circle (2pt) node[anchor=west]{};
\filldraw[black] (3,-1.7) circle (0pt) node[anchor=west]{$l$};
\filldraw[black] (0,-3.4) circle (0pt) node[anchor=north]{$m$};

\draw (4,0) -- (2,3.4) -- (-2,3.4) -- (-4,0) -- (-2,-3.4) -- (2,-3.4) -- (4,0) ;
\draw[blue,dashed] (2,3.4) --  (-4,0) -- (2,-3.4) -- (2,3.4);
\draw[red, line cap=round, line width=2.5pt, dash pattern=on 0pt off 3\pgflinewidth] (2,-3.4) --  (3,1.7) -- (-2,3.4);

\end{tikzpicture}
   \caption{Some lines of $\MA$ spanned by points of an absolute hexagon.}
\label{fig:MovingHexagon}
\end{figure}
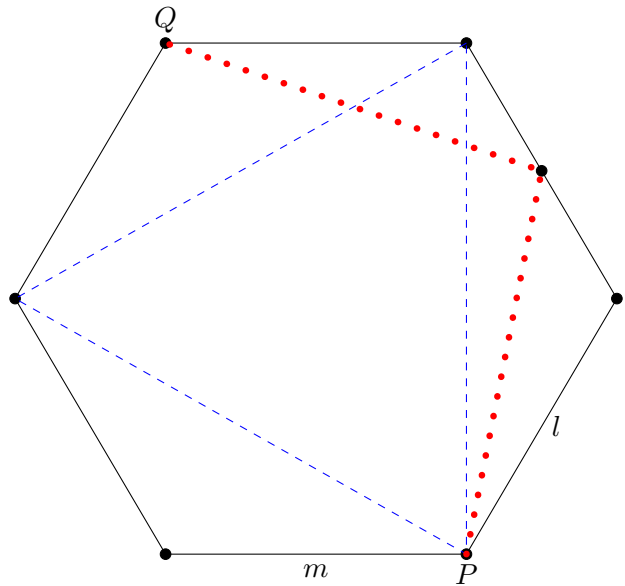

We now investigate how many absolute points a line of $\MA$ contains.
\begin{lemma}\label{lem:movingLine}
    If $l$ is a moving line containing at least two absolute points, it contains exactly $k+1$ absolute points.
\end{lemma}

\begin{proof}
Let $P$ be an absolute point and let $P \omega$ be its associated plane. By~\cite[\$8.2.3]{tits1959trialite}, the number of absolute lines going through $P$ and inside of $P \omega$ is equal to $k+1$. There are no other absolute points in $P \omega$ than the ones contained in those $k+1$ lines. This can be shown by a counting argument using~\cite[\$8.2.4 and \$8.2.6]{tits1959trialite}.

Now, let $l$ be a moving line containing at least two absolute points $P_1$ and $P_2$. Then, there exists an absolute hexagon $\mathbf{H}$ such that $P_1$ and $P_2$ are vertices of $\mathbf{H}$. We claim that $P_1$ and $P_2$ are neither adjacent nor opposite in $\mathbf{H}$. Indeed, they cannot be adjacent, else $l$ would be an absolute line, and Lemma \ref{lem:oppositeVertices} tells us that they cannot be opposite. 

Let thus $P$ be the unique vertex of $\mathbf{H}$ at distance one from both $P_1$ and $P_2$. Then $l$ is inside the plane $P \omega$ and $l$ does not contain $P$. Hence the absolute points of $l$ are in one to one correspondence with the absolute lines going through $P$.
\end{proof}

As an immediate corollary, we have that all lines in $M\Gamma$ contain exactly $k+1$ absolute points. In fact $\MA$ is flag-transitive.

\begin{lemma}\label{lem:flagTrans}
    The moving absolute geometry $M\Gamma$ is flag-transitive.
\end{lemma}

\begin{proof}
    Take two flags $(P_1,l_1)$ and $(P_2,l_2)$ of $\MA$.
    As we showed in Lemma \ref{lem:movingLine}, a moving line can always be placed in an absolute hexagon $\mathbf{H}$ as a line between two vertices at distance $2$. 
    By~\cite[Theorem 6.2.5.]{tits1959trialite}, we know that the group $G$ of collineations of $\mathbf{P}$ preserving $\mathbf{Q}$ and $\alpha$ acts transitively on triples of absolute points $(Q,P,Q')$ such that $(PQ)$ and $(PQ')$ are absolute lines.
    So there exist pairs of absolute points $(Q_1,Q_1')$ and $(Q_2,Q_2')$ such that $(P_iQ_i) = l_i$ (with $i=1,2$) and a collineation mapping $(P_1,Q_1',Q_1)$ to $(P_2,Q_2',Q_2)$ and hence also $l_1$ to $l_2$. 
    Therefore the group $G$ acts transitively on the flags of $\MA$.
\end{proof}

We now compute the number of lines of $\MA$ passing through a given absolute point $P$.

\begin{lemma}\label{lem:movingLinesPerPoint}
    Let $P$ be an absolute point. There are $(k+1)f^2$ moving lines passing through $P$.
\end{lemma}

\begin{proof}
    Let $x$ be the number of moving lines through $P$ and let $y$ be the number of absolute vertices at distance $2$ of $P$ in the absolute geometry $\Gamma$. We first claim that $kx = y$. Indeed, if $Q$ is an absolute point at distance $2$ from $P$, then there exists an absolute point $R$ such that $(PR)$ and $(RQ)$ are absolute lines. Hence all three points are in $ R \omega$ and the line $(PQ)$ is a line of $\MA$. But any other point $Q'$ on $(PQ)$ would yield the same line since $(PQ') = (PQ)$. By Lemma~\ref{lem:movingLine} there are $k+1$ points on any line of $\MA$. Hence this proves the claim.

    It now suffices to find out what $y$ is. There are $k+1$ absolute lines containing $P$. Each of these lines contains $f+1$ absolute points. So we have $(k+1)f$ neighbours of $P$ in $\Gamma$. Through each of these neighbours pass $k$ new absolute lines which all contain $f$ new absolute points. We thus counted $k(k+1)f^2$ potential absolute points at distance $2$ of $P$ in $\Gamma$. Since the girth of $\Gamma$ is $6$, we could not possibly have counted points twice, so that $y = k(k+1)f^2$ and $x = (k+1)f^2$.
\end{proof}

Using flag-transitivity together with the previous results, we can now count the number of lines of $\MA$.

\begin{lemma}\cite[\$8.2.4]{tits1959trialite}\label{lem:absolutePoints}
There are $(k^2f^2 +kf +1)(f+1)$ points in $\MA$.
\end{lemma}

\begin{lemma} \label{lem:movingLines}
    There are $(k^2f^2 +kf +1)(f+1)f^2$ lines in $\MA$.
\end{lemma}

\begin{proof}
    By Lemma~\ref{lem:absolutePoints}, there are $(k^2f^2 +kf +1)(f+1)$ absolute points. The number of lines of $\MA$ can be computed by multiplying the number of points by the number of lines per point (that is $(k+1)f^2$ by Lemma~\ref{lem:movingLinesPerPoint}) and dividing by the number of points per line (that is $(k+1)$ by Lemma~\ref{lem:movingLine}). 
\end{proof}
\begin{lemma}
    Every line of $\MA$ is contained in exactly one special plane and every special plane contains $f^2$ lines of $\MA$.
\end{lemma}
\begin{proof}
    Let $P$ be an absolute point and $P \omega$ be its associated plane. We claim that all lines of $P \omega$, except for those containing $P$, are lines of $\MA$. 
    Indeed, a line $l$ of $P \omega$ that does not contain $p$ must intersect every line of $P \omega$ containing $P$. This means that $l$ contains $k+1$ absolute points and $l$ cannot be an absolute line, else we would have found a triangle of absolute lines in $\Gamma$.

    This gives us a way to associate to every absolute point $f^2$ lines of $\MA$. Clearly, we can count every line of $\MA$ in this way. Moreover, since we know that, by Lemma~\ref{lem:movingLines}, there are $f^2$ times more lines than points in $\MA$, we can deduce that we never count a line twice in this way, which proves the lemma.
\end{proof}

At this point, we remark that although it may seem natural to think that the triality $\alpha$ acts on the set of lines of $\MA$, that is actually not the case. Indeed, if $l$ is a line of $\MA$, then $\alpha(l)$ is still a line but it never contains more than one absolute point. To see this, let $P \omega$ be the unique special plane containing $l$ and suppose that $\alpha(l)$ also contains two absolute points. Then $\alpha(l)$ is also contained in a unique special plane $ R \omega$ for some absolute point $R$. The line $\alpha^2(l)$ must then be $(PQ)$. But $Q$ is a point of $l$, since $l = \alpha^2(\alpha(l))$. Therefore $(PQ)$ is a line of $P \omega$ containing $P$ and $Q$, and must then be an absolute line, a contradiction.

We now compute an upper bound for the number of lines at distance $4$ or less from a given line $l$ of $\MA$. This will be used later to show the existence of two lines $l$ and $l'$ of $\MA$ such that the distance between 
$l$ and $l'$ is six in the incidence graph of $\MA$.

\begin{lemma} \label{lem:counting}
    Let $l$ be a moving line and let $C = (k+1)f^2$ be the number of moving lines through a point. There are at most $1 + (k+1) (C-1) + k (k+1) (C-1)^2$ moving lines at distance $4$ or less from $l$ in the incidence graph of $\MA$. 
\end{lemma}

\begin{proof}
    Denote by $A(l,k)$ the number of elements of $\MA$ at distance exactly $k$ from $l$ in the incidence graph of $\MA$. If $k$ is odd, these elements will all be points and if $k$ is even they will be lines.
    Hence, we need to compute the number $A(l,0) + A(l,2) + A(l,4)$. Obviously, we have that $A(l,0) = 1$. It is also clear that $A(l,2) = (k+1)(C-1)$. Indeed, there are exactly $(k+1)$ absolute points on $l$ and there are $C$ lines through each of them, one of them being $l$ every time. After that it becomes more difficult to get precise numbers as there are triangles in $\MA$, but we can ignore the triangles and still get upper bounds. Let us estimate $A(l,3)$. We have $A(l,2)$ lines at distance two from $l$ and each of these lines yields $k$ absolute points at distance $3$ from $l$. This gives us $k(k+1)(C-1)$ potential points. Hence, $A(l,3) \leq k (k+1) (C-1)$. We proceed the same way for computing $A(l,4)$. Through each one of the points of $A(l,3)$ passes $C$ moving lines. That gives us $k (k+1)(C-1)^2$ potential lines in $A(l,4)$, which is what we need to conclude the proof. 
\end{proof}

We need two more Lemmas before proving the main Theorem, that is Theorem~\ref{main}. The first Lemma is used to show that some configuration of points are not far from each other in $\MA$ and the second one is the key ingredient in the proof that $d_l = 6$ in $\MA$.

\begin{lemma} \label{lem:incidentVertices}
    Let $l$ be an absolute line and $P_1$ and $P_2$ be vertices on $l$. There exists an absolute vertex $P$ not on $l$ such that $(PP_1)$ and $(PP_2)$ are moving lines in $\mathbf{Q}$.
\end{lemma}
\begin{proof}
    The line $l$ must contain at least another point $P_3$ which is then an absolute point. Let $\mathbf{H}$ be an hexagon containing $P_1$ and $P_3$. Since $P_1$ and $P_3$ are adjacent in $\mathbf{H}$, there is a unique vertex $P \neq P_1$ or $P_3$ in $\mathbf{H}$. But then $P_1,P_2,P_3$ and $Q$ are all in the plane $P \omega_3$ and the lines $(PP1)$ and $(PP_2)$ cannot be absolute lines, else there would be an absolute triangle. 
\end{proof}

\begin{lemma} \label{lem:dist6}
    There exist two moving lines $l$ and $l'$ that are at distance $6$ or more in the incidence graph of $\MA$.
\end{lemma}

\begin{proof}
    In the case of $f = k^3$, this follows from Lemma \ref{lem:counting} except for $f = 8$ and $k=2$, in which case we can use {\sc Magma}~\cite{magma} or some other program to verify the claim. Indeed, by Lemma \ref{lem:movingLines}, we know that there are a total of $A = (k^2f^2 +kf +1)(f+1)f^2$ lines in $\MA$. We also know, by Lemma \ref{lem:counting}, that if we choose a line $l$ as a base line, there are at most $B =1 + (k+1) (C-1) + k (k+1) (C-1)^2$, where $C = (k+1)f^2$, lines at distance $4$ or less from $l$.  Consider $A$ and $B$ as polynomials in $f$. When $f = k^3$, the leading term for $A$ is of order $5$ while the leading term for $B$ is of order $4$. A straighforward analysis then shows that $A > B$ as soon as $k \geq 3$.  See Lemma \ref{lem:stab} and the comments thereafter for a way to construct $\MA$ on a computer, and thus verify that the statement holds in the small cases too.

    On the other hand, if $\sigma$ is the identity, we need to proceed differently since in that case $B > A$ except maybe when $f =k$ is small enough. We will explicitly show that there exists a pair of lines at distance $6$. In this case, it is known that the absolute points are exactly the intersection of $\mathbf{Q}$ with the hyperplane of equation $X_3 + X_7 = 0$. Therefore, by substituting $X_7$ by $-X_3$, we can work in a projective space of dimension $6$ instead and use Grassmann coordinates to characterize the absolute lines by the set of equations (\ref{Grassmann}) described in the previous section. Let $e_i$ be the points having all homogeneous coordinates equal to $0$ except for the $i^{th}$ one that can be set to $1$. Using the trilinear form $\mathcal{T}$ with equation (\ref{trilinear}), we can check that all the $e_i$, except for $i= 3$ or $7$, are absolute points. Using the equations (\ref{Grassmann}), we also check that the path $e_0 -e_5 -e_2-e_4-e_1-e_6-e_0$ forms an hexagon of absolute lines and points. We claim that the lines $l = (e_5e_6)$ and $l'= (e_1e_2)$ are lines of $\MA$ at distance $6$ in the incidence graph of $\MA$. In other words, we have to show that there is no line of $\MA$ connecting a point of $l$ to a point of $l'$. Easy computations show that any line connecting $l$ to $l'$ satisfies the set of equations (\ref{Grassmann}). This means that any such line is either an absolute line or is not a line of the conic $\mathbf{Q}$.
\end{proof}

We are now ready to prove the main theorem of this section.

\begin{theorem}\label{main}
    Let $\mathbf{Q}$ be a quadric in a 7-dimensional projective space $\mathbf{P}$ over a finite field $\mathbb{F}$ of cardinality $f$ and let $\Gamma$ be its associated geometry with diagram $D_4$. Let $\sigma$ be an automorphism of $\mathbb{F}$ of order $o(\sigma) = 1$ or $3$ and let $\alpha$ be a triality of $\Gamma$ of type $(I_\sigma)$. The moving absolute geometry $M\Gamma_\alpha$ is a flag-transitive geometry with the following Buekenhout diagram where $f=k^{o(\sigma)}$.

    \begin{center}

    \begin{tikzpicture}
    
    \filldraw[black] (-3,0) circle (2pt)  node[anchor=north]{$k$};
    \filldraw[black] (3,0) circle (2pt)  node[anchor=north]{$(k+1)f^2-1$};
    \draw (-3,0) -- (3,0) node [midway, above = 3pt, fill=white]{$5 \; \; \; \; \; 3 \; \; \; \; \; 6$};
    \filldraw[black] (3,-1.2) circle (0pt) node[anchor=south]{$(k^2f^2 +kf +1)(f+1)f^2$} ;
    \filldraw[black] (-3,-1.2) circle (0pt) node[anchor=south]{$(k^2f^2 +kf +1)(f+1)$} ;
    \end{tikzpicture}
    \end{center}
    Moreover, the group $G$ acts flag-transitively on $\MA_\alpha$ where $G$ is $G_2(k)$ when $o(\sigma) = 1$ and $^3D_4(k)$ when $o(\sigma)=3$.
\end{theorem}

\begin{proof}
By Lemma~\ref{lem:flagTrans}, $\MA$ is flag-transitive.
    We already showed that the number of points, lines, points per line and lines per point are as indicated in the diagram (see Lemmas~\ref{lem:absolutePoints},~\ref{lem:movingLines},~\ref{lem:movingLine},~\ref{lem:movingLinesPerPoint} respectively). It remains thus to show that $d_P = 5, g = 3$ and $d_L = 6$.
    Let $P_1$ and $P_2$ be two points of $\mathbf{Q}$. Then, there exists an absolute hexagon $\mathbf{H}$ such that $P_1$ and $P_2$ are vertices of $\mathbf{H}$. Fix this hexagon once and for all. We already showed the existence of the dashed line in figure \ref{fig:MovingHexagon}. Therefore, there are triangles in $\MA$. As all the lines are lines of a projective space, no two lines have two points in common and the gonality $g$ of $M\Gamma$ must then be $3$.

    We claim that $P_1$ and $P_2$ are at distance at most $4$ in the incidence graph of $\MA$. Suppose first that $P_1$ and $P_2$ are adjacent in $\mathbf{H}$. Then, the distance between $P_1$ and $P_2$ is $4$ by Lemma~\ref{lem:incidentVertices}. If they are opposite vertices of $\mathbf{H}$, their distance in the incidence graph is also $4$. Indeed, let $l$ be a line of $\mathbf{H}$ at distance $3$ from both $P_1$ and $P_2$. Then, there is a point $P$ on $l$ which is not a vertex of $\mathbf{H}$, since even in the smallest case of $f=2$ there are $3$ points per line. Since every point on an absolute line is an absolute point, $P$ is an absolute point. Moreover, the lines $P_1P$ and $PP_2$ are lines of $\mathbf{Q}$ since any $3$ consecutive vertices of an  hexagon are in the plane spanned by the two lines emanating from the middle vertex. The lines $P_1P$ and $PP_2$ are moving lines, else we would find triangles in the absolute. Such a path of length $2$ between two absolute points is illustrated by the dotted lines in Figure \ref{fig:MovingHexagon}.

    This shows that the maximal distance between two absolute vertices in the incidence graph of $M\Gamma$ is $4$. This implies that the maximal distance between a vertex $P$ and a moving line $l$ is $3$ or $5$.
 
    It cannot be $3$.
    Indeed, take a point $P$ to be a vertex of $\mathbf{H}$ and $l$ to be the line between $P_1$ and $P_2$ where $P_1$ is a neighbor of $P$ in $\mathbf{H}$ and $P_2$ is the opposite vertex of $P$ in $\mathbf{H}$. Then, by the one or all axiom, the only line joining $P$ to a point of $l$ is $PP_1$ (since $PP2$ cannot be a line of $\mathbf{Q}$). Therefore, there is no line of $\MA$ connecting $P$ to a point of $l$ and thus their distance is strictly greater than $3$. This proves that $d_P = 5$.

    It only remains to show that $d_L = 6$. Since $d_P = 5$, $d_L$ can only be $5$ or $6$. Lemma \ref{lem:dist6} shows the existence of two lines at distance $6$ from each other. Hence $d_L = 6$.
\end{proof}

We end this section by showing how to construct the moving absolute geometries $\MA$ associated to the quadric $\mathbf{Q}$ using a computer. The group $G$ is the group of collineations preserving the triality $\alpha$. Depending on whether $\sigma$ is the identity or not, the group $G$ is isomorphic to $G_2(q)$ or to $^3D_4(q)$ (see \cite[section 6]{tits1959trialite} for more details). 

\begin{lemma} \label{lem:stab}
    Let $H$ be the stabilizer in $G$ of an absolute point $P$. Then $H$ acts transitively on the $f^2$ lines of $\MA$ contained in $P \omega$.
\end{lemma}

\begin{proof}
    We already mentioned in Lemma~\ref{lem:flagTrans} that $G$ acts transitively on triples of absolute points $(Q,Q',Q'')$ such that $(Q',Q)$ and $(Q'',Q')$ are absolute lines. Then, if we fix an absolute point $P$ and two absolute lines $l_1$ and $l_2$ containing $P$, we immediately deduce that $H$ still acts transitively on those of the above triples with $Q'= P$, $Q \in l_1$ and $Q'' \in l_2 $. It suffices to notice that any of the $f^2$ lines of $\MA$ can be written as $(QQ'')$ for some $Q \in l_1$ and $Q'' \in l_2$.
\end{proof}

The stabiliser in $G$ of an absolute point $P$ is the same whether we consider the classical absolute or our moving absolute geometry. Therefore, it is usually very easy to find such stabilizer for any known representations of $G$. By Lemma \ref{lem:stab} we know that the stabilizer $K'$ of a line of $\MA$ in $P \omega$ is a subgroup of index $f^2$ of $H$. We can thus find $K'$ by looking at the list of maximal subgroups of $H$. Let $l$ by a line of $\MA$ incident to $P$ so that $\{P,l\}$ is a flag. Then, since $G$ acts transitively of $\MA$, the stabilizer of $l$ must be a conjugate of $K'$. We can then look for a suitable conjugate of $K'$, namely a conjugate $K$ of $K'$ such that $H \cap K$ has the right cardinality (i.e. $\frac{|H|}{|H \cap K|} = (k+1)f^2$). The moving absolute geometry $\MA$ is then obtained as the coset geometry $\Gamma(G, \{H,K\})$.

\subsection{Maps of class III}

In their recent article \cite{leemans2022incidence}, Leemans and Stokes showed how to construct reflexible maps having trialities but no dualities (i.e maps of class III), using the simple group $L_2(q^3)$ with $q = p^n$ for a prime number $p$.

From one of these maps they also show how to construct an incidence geometry $\Delta$ having as elements the vertices, edges, faces and Petrie paths of the map. A face and a Petrie path are considered incident if they share a an edge and the rest of the incidence relation is given by symmetrized inclusion. They then show that, since $\Delta$ is constructed from a class III map, it admits trialities but no dualities in the sens of incidence geometry also. We recall the precise construction of such geometries.

Let $G = L_2(q^3)$ and let $\rho_0,\rho_1,\rho_2$ be three involutions generating $G$ and let $\alpha \in \Out(G)$ be an outer automorphism of order $3$ such that

\begin{enumerate}
    \item $\alpha$ cyclically permutes $\rho_0, \rho_2$ and $\rho_0\rho_2$,
    \item $\alpha$ fixes $\rho_1$,
    \item $\rho_0$ and $\rho_2$ commute,
    \item $\langle \rho_0, \rho_1, \rho_2 \rangle = G$, and
    \item there is no element of $\Aut(L_2(q^3))$ that swaps $\rho_0$ and $\rho_2$ and fixes $\rho_1$.
\end{enumerate}

For any choices of $q$,$\rho_0,\rho_1,\rho_2$ and $\alpha$ satisfying the above conditions Leemans and Stokes construct a coset geometry $\Delta = \Gamma(G;\{G_0,G_1,G_2,G_3\})$ where 
\begin{itemize}
    \item $G_0 = \langle \rho_0, \rho_1 \rangle$
    \item $G_1 = \langle \rho_0, \rho_2 \rangle$
    \item $G_2 = \langle \rho_1, \rho_2 \rangle$
    \item $G_3 = \langle \rho_1, \rho_0\rho_2 \rangle$
\end{itemize}

We start by analysing the classical absolute geometry of $\Delta$.

\begin{theorem}\label{thm: absoluteMaps}
    Let $\Delta = \Gamma(G;\{G_0,G_1,G_2,G_3\})$ be as above and let $\Delta_\alpha$ be its absolute with respect to the triality $\alpha$. Then $\Delta_\alpha$ is a graph which is the disjoint union of $\frac{|L_2(q)|}{2}$ paths of length $2$.
\end{theorem}
\begin{proof}
    We first show that $\Delta_\alpha$ is a union of path of length $2$ and then we will compute the number of edges of $\Delta_\alpha$.

    Suppose there is at least one edge $e$ fixed by $\alpha$ and label its endpoints by $v_1$ and $v_2$. First notice that $v_1$ and $v_2$ are absolute points. Indeed, since $\alpha$ fixes $e$ and preserves incidence, we get that $v_i * e$ implies that $\alpha(v_i) * e$ and $\alpha^2(v_i) * e$ for $i =1,2$. Label $\alpha(v_i)$ by $F_i$ and $\alpha^2(v_i)$ by $P_i$.
    
    Figure \ref{fig:fixedEdge} shows a local picture around the fixed edge $e$. This picture is to be understood as being drawn on the underlying surface of the original map of class III. 

    Among all edges coming out of $v_1$ or $v_2$ there is exactly one edge that is also incident to a triple ${v_i,F_i,P_i}$. We label that edge by $e'$. It is easy to see that $e'$ must also be fixed by $\alpha$. We thus showed that fixed edges appear in pairs. Label the other endpoint of $e'$ by $v_3$. It remains to show that no other edges among all the edges coming out of $v_1$,$v_2$ and $v_3$ can be fixed by $\alpha$.

    Let $F_3$ be the face containing $e'$ which is not $F_1$. Then, we notice that the set of all the edges around a vertex $v_i$ must be sent by $\alpha$ to the set of edges of $F_3$. This is because if an edge $x$ is incident to a vertex $v_i$, then $\alpha(x)$ must be incident to $\alpha(v_i) = F_i$. In the case of $v_1$ this already proves the claim, since the only edges of $F_1$ incident to $v_1$ are $e$ and $e'$. For $v_2$, there remains one edge $x \neq e$ incident to both $v_2$ and $F_2$. But $x$ is incident to $P_1$ so $\alpha(x)$ must be incident to $\alpha(P_1) = v_1$ and thus cannot be fixed.
    The case of $v_3$ is identical to the one of $v_2$. There is only one edge $y \neq e'$ which is incident to both $v_3$ and $F_3$. But $y$ is incident to $P_1$ again, so $\alpha(y)$ must be incident to $v_1$ and thus cannot be fixed. 
    This concludes the proof that $\Delta_\alpha$ is a union of paths of length 2.
    
    We will now show that the number of edges of $\Delta_\alpha$ is equal to $|L_2(q)|$. To do so we will show that there is a one-to-one correspondence between edges of $\Delta_\alpha$ and fixed points of the action of $\alpha$ on $L_2(q^3)$. The automorphism $\alpha$ being a field automorphism, it fixes elementwise a subfield subgroup of $L_2(q^3)$ isomorphic to $L_2(q)$.
    Here we use the fact that $\Gamma$ is a coset geometry. Since $\rho_0$ and $\rho_2$ commute, the maximal parabolic subgroup $G_1 = \{Id_G, \rho_0, \rho_2, \rho_0\rho_2\} = \{Id_G, \rho_0, \alpha(\rho_0), \alpha^2(\rho_0)\}$.
    The edges of $\Gamma$ are thus of the form $G_1 \cdot x = \{x, \rho_0 x, \alpha(\rho_0) x, \alpha^2(\rho_0) x\}$ for some $x \in G$. Suppose such a $G_1 \cdot x$ is fixed by $\alpha$. Then $G_1 \cdot x$ must be a union of $\alpha$-orbits. Since the orbits of $\alpha$ on $G$ have length $1$ or $3$, $G_1 \cdot x$ can only be made of $4$ fixed points or $1$ fixed point and an orbit of size $3$. The former case is easily seen to be impossible as it would imply that $G_1$ is also made of $4$ fixed points and thus $\alpha$ should be the identity. Therefore, every fixed $G_1 \cdot x$ contains exactly $1$ fixed point of $\alpha$. And of course, if $x$ is fixed by $\alpha$, then $G_1 \cdot x$ is also fixed as $x \in G_1 \cdot x \cap \alpha(G_1 \cdot x)$.
\end{proof}

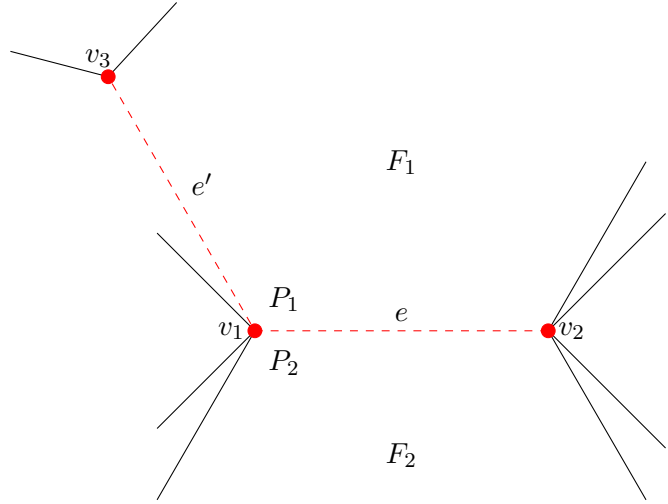
\begin{figure}
\centering
\begin{tikzpicture}[scale = 1.3]

\draw[red,dashed] (3,0) -- (0,0) -- (-1.5,2.6) ;
\draw[black] (0,0) -- (-1,1) ;
\draw[black] (0,0) -- (-1,-1) ;
\draw[black] (0,0) -- (-1,-1.73) ;

\draw[black] (3,0) -- (4,1.73) ;
\draw[black] (3,0) -- (4.2,1.2) ;
\draw[black] (3,0) -- (4.2,-1.2) ;
\draw[black] (3,0) -- (4,-1.73) ;

\draw[black] (-1.5,2.6) -- (-2.5,2.86) ;
\draw[black] (-1.5,2.6) -- (-0.8,3.36) ;

\filldraw[red] (0,0) circle (2pt)  ;
\filldraw[red] (3,0) circle (2pt) ;
\filldraw[red] (-1.5,2.6) circle (2pt) ;
\filldraw[black] (0,0) circle (0pt) node[anchor=east]{$v_1$} ;
\filldraw[black] (3,0) circle (0pt) node[anchor=west]{$v_2$} ;
\filldraw[black] (-1.6,2.6) circle (0pt) node[anchor=south]{$v_3$} ;

\filldraw[black] (-0.75,1.5) circle (0pt) node[anchor=west]{$e'$} ;
\filldraw[black] (1.5,0) circle (0pt) node[anchor=south]{$e$} ;

\filldraw[black] (1.5,1.5) circle (0pt) node[anchor=south]{$F_1$} ;
\filldraw[black] (1.5,-1.5) circle (0pt) node[anchor=south]{$F_2$} ;

\filldraw[black] (0.3,0.1) circle (0pt) node[anchor=south]{$P_1$} ;
\filldraw[black] (0.3,-0.1) circle (0pt) node[anchor=north]{$P_2$} ;

\end{tikzpicture}
   \caption{Local picture around a fixed edge. The dashed edges and their endpoints correspond to elements of the absolute geometry.}
\label{fig:fixedEdge}
\end{figure}

This entirely characterises the absolute geometries for the incidence geometries obtained from class III maps constructed using $L_2(q)$ in~\cite{leemans2022incidence}. Remark that in this case the absolute geometry is independent of the choices of generators $\rho_0,\rho_1$ and $\rho_2$ and of the triality $\alpha$. 

The core of the issue here is that the residues of $\Delta$ are too small for the absolute geometry to be interesting. The moving absolute geometry is not so controlled by the size of the residue, as we will show with some examples below.

We have written a {\sc Magma} program that computes the moving absolute geometry $M\Delta_\alpha$ from $L_2(q^3)$ together with a generating set $\{\rho_0,\rho_1, \rho_2\}$. The algorithm used by the program follows the five steps described below: 
\begin{enumerate}[1.]
    \item Compute the coset geometry $\Gamma(G, \{G_0,G_1,G_2,G_3\})$.
    \item Find a triality $\alpha \in \Aut(G)$: 
 The group $G = L_2(q^3)$ is given as a permutation group over $q^3+1$ points. The group $G$ is thus a subgroup of $\Sym(q^3+1)$. Construct the centraliser $C = C_{\Sym(q^3+1)}(G)$ of $G$ in $\Sym(q^3+1)$. Elements of $C$ can be seen as automorphisms of $G$ via their action by conjugation. Since $\Aut(G)$ is an extension of order $3$ of $G$, as long as $C \neq G$ we know that $C = \Aut(G)$. Look for an element of order $3$ in $C\setminus G$ that centralizes $\rho_1$ and permutes $\rho_0,\rho_2$ and $\rho_0\rho_2$. That will be the triality $\alpha$.

    \item Compute the absolute points: for each coset $G_0 \cdot x$, check if  $G_0 \cdot x \cap (G_0 \cdot x)^\alpha$ is empty or not. Keep the ones for which the intersection is not empty. Note that this is sufficient as $\alpha$ is a triality and thus $v * \alpha(v)$ implies $\alpha^{-1}(v) = \alpha^2(v) * v$. 

    \item Compute the edges: given a coset $G_1 \cdot x$, check if $\alpha x \alpha ^{-1} x^{-1}$ is in $G_1$ or not. Keep the ones for which $\alpha x \alpha ^{-1} x^{-1}$ is not in $G_1$. Indeed, $(G_1 \cdot x)^\alpha = \alpha G_1 \cdot x \alpha ^{-1} = \alpha G_1 \cdot (\alpha ^{-1} \alpha) x \alpha ^{-1} = G_1 \cdot \alpha x \alpha ^{-1}$.

    \item Match each moving edge with its endpoints and create a graph.
\end{enumerate}

\begin{figure}
\centering
\begin{tikzpicture}[scale = 2.0]

\draw[black] (0,0) -- (0,1) -- (-1.5,1.5) --(-0.5,1.5) -- (1,1) --(0,0);
\draw[black] (0,1) -- (1,1);
\draw[black] (0,0) -- (-1.5,0.5)--(-1.5,1.5);
\draw[black] (-1.5,0.5) -- (-0.5,1.5);

\filldraw[black] (0,0) circle (1.5pt) ;
\filldraw[black] (0,1) circle (1.5pt) ;
\filldraw[black] (1,1) circle (1.5pt) ;

\filldraw[black] (-1.5,0.5) circle (1.5pt) ;
\filldraw[black] (-1.5,1.5) circle (1.5pt) ;
\filldraw[black] (-0.5,1.5) circle (1.5pt) ;

\draw[blue,dashed] (0.0,2.5) -- (-1.5,1.5);
\draw[blue,dashed] (0.0,2.5) -- (1,1);

\draw[blue,dashed] (1.5,0.2) -- (-0.5,1.5);
\draw[blue,dashed](1.5,0.2) -- (0,0);

\draw[blue,dashed] (-2,-0.5) -- (-1.5,0.5);
\draw[blue,dashed] (-2,-0.5) -- (0,1);

\filldraw[black] (0.0,2.5) circle (1.5pt) ;
\filldraw[black] (1.5,0.2) circle (1.5pt) ;
\filldraw[black] (-2,-0.5) circle (1.5pt) ;

\end{tikzpicture}
   \caption{The two absolute geometries for $G = L_2(8)$. The classical absolute consists of the  vertices and the dashed edges and the moving absolute geometry of the vertices and the full edges.}
\label{fig:L2(8)}
\end{figure}
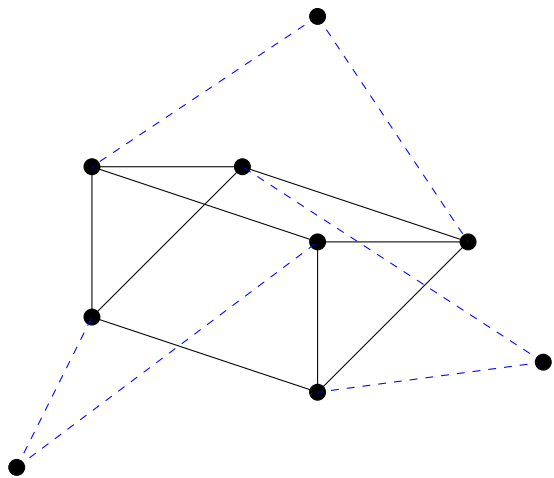

We end this article by mentioning a few examples of the moving absolute geometries computed by the above algorithm. Remark that the moving absolute geometry depends not only of the cardinality $q$ of the underlying field $\mathbb{F}$ but also of the choice of the generating set $\{\rho_0,\rho_1,\rho_2\}$ and of the triality $\alpha$. We also adopt the convention that if $M \Delta_\alpha$ contains isolated vertices, we remove them.

\begin{enumerate}
    \item For $G = L_2(2^3)$ there is only one choice, up to conjugation, of generating set $\{\rho_0,\rho_1, \rho_2\}$. We know that $\Delta_\alpha$ is always a disjoint union of $3$ paths of length $2$. In this case $M\Delta_\alpha$ is a prism with triangular basis. In Figure~\ref{fig:L2(8)} we show $M\Delta_\alpha$ and the dashed lines show how $\Delta_\alpha$ attaches to $M\Delta_\alpha$ inside of $\Delta$.
    \item For $G = L_2(4^3)$, there is a moving absolute geometry $M\Delta_\alpha$ which is a disjoint union of $15$ edges and $12$ pentagons. It has $90$ vertices if degree $3$ and $75$ edges, its girth is equal to $5$, diameter equal to $7$ and its automorphism group is isomorphic to $2 \times A_5$.
    \item For $G = L_2(5^3)$ the following moving absolute geometries $M\Delta_\alpha$ appear:
    \begin{enumerate}[1.]

    \item A graph with $60$ vertices, with girth equal to $3$, diameter equal to $8$ and automorphism group isomorphic to $2 \times A_5$. This graph is vertex transitive.
    \item A graph with $30$ vertices of degree $4$, $60$ edges which is arc-transitive and has Buekenhout diagram:
    
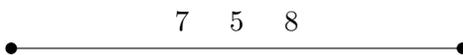
\begin{center}
    \begin{tikzpicture}
    
   \filldraw[black] (-3,0) circle (2pt)  node[anchor=north]{};
   \filldraw[black] (3,0) circle (2pt)  node[anchor=north]{};
    \draw (-3,0) -- (3,0) node [midway, above = 3pt, fill=white]{$7\; \; \; \; \; 5 \; \; \; \; \; 8$};
    \end{tikzpicture}
    \end{center}
    
    Its automorphism group is $2 \times \Sym(5)$.
    \end{enumerate}

    \item For $G= L_2(7^3)$, the following moving absolute geometries $M\Delta_\alpha$ appear :
    \begin{enumerate}[1.]
        \item A graph on $84$ vertices of degree $4$ which is admits a perfect matching. 
        \item A graph on $168$ vertices of degree $4$. It is connected, has girth equal to $3$, diameter equal to $9$ and automorphism group isomorphic to $L_2(7)$.
    \end{enumerate}

    \item For $G = L_2(9^3)$, there is a moving absolute geometries $M\Delta_\alpha$ which is a graph on $180$ vertices of degree $4$ and $360$ edges. It has Buekenhout diagram:

    \begin{center}

    \begin{tikzpicture}
    
   \filldraw[black] (-3,0) circle (2pt)  node[anchor=north]{};
   \filldraw[black] (3,0) circle (2pt)  node[anchor=north]{};
    \draw (-3,0) -- (3,0) node [midway, above = 3pt, fill=white]{$13 \; \; \; \; \; 5 \; \; \; \; \; 13$};
    \end{tikzpicture}
    \end{center}
        
    \end{enumerate}

\end{document}